\documentclass[reqno]{amsart}
\usepackage[active]{srcltx}
\usepackage{array}
\usepackage{amsmath,amscd,amssymb}
\usepackage{latexsym}
\usepackage{epsfig}
\usepackage[all]{xy}
\usepackage[T1]{fontenc}
\usepackage{color}

\newtheorem{theorem}{Theorem}[section] 
\newtheorem{lemma}[theorem]{Lemma}
\newtheorem{proposition}[theorem]{Proposition}
\newtheorem{corollary}[theorem]{Corollary}

\newtheorem{definition}[theorem]{Definition}

\newcommand{\CC}{\mathbb C}

\newcommand{\HH}{\mathbb H}
\newcommand{\NN}{\mathbb N}
\newcommand{\PP}{\mathbb P}

\newcommand{\RR}{\mathbb R}
\newcommand{\ZZ}{\mathbb Z}

\newcommand{\cB}{\mathcal B}

\newcommand{\cD}{\mathcal D}

\newcommand{\cH}{\mathcal H}

\newcommand{\cL}{\mathcal L}

\newcommand{\cX}{\mathcal X}
\newcommand{\cY}{\mathcal Y}
\newcommand{\cZ}{\mathcal Z}

\newcommand{\Prod}{\prod\limits}

\newcommand{\Sp}{\mathop{\mathrm {Sp}}\nolimits}

\newcommand{\SL}{\mathop{\mathrm {SL}}\nolimits}

\newcommand{\Orth}{\mathop{\null\mathrm {O}}\nolimits}

\newcommand{\Lift}{\mathop{\mathrm {Lift}}\nolimits}

\renewcommand{\Im}{\mathop{\mathrm {Im}}\nolimits}

\newcommand{\rank}{\mathop{\mathrm {rank}}\nolimits}

\newcommand{\id}{\mathop{\mathrm {id}}\nolimits}

\newcommand{\gz}{\mathfrak z}
\newcommand{\gG}{\mathfrak G}
\newcommand{\mg}{\mathfrak g}
\newcommand{\latt}[1]{{\langle{#1}\rangle}}

\newcommand{\Kthree}{\mathop{\mathrm {K3}}\nolimits}

\begin{document}

\title{$24$ faces of the Borcherds modular form $\Phi_{12}$}
\author{Valery Gritsenko}
\date{}
\maketitle
\begin{abstract}
The fake monster Lie algebra is determined by the Borcherds function 
$\Phi_{12}$ which is the reflective 
modular form of the minimal possible weight  with respect to $\Orth^+(II_{2,26})$.
We prove that the  first non-zero Fourier--Jacobi coefficient of $\Phi_{12}$
in any of $23$ Niemeier cusps is equal to the Weyl--Kac denominator function 
of the affine Lie algebra of the root system of the corresponding Niemeier lattice.
This is an automorphic answer (in the case of the fake monster Lie algebra)
on the old question of Frenkel and Feingold (1983)
about possible relations between  hyperbolic Kac--Moody algebras,
Siegel modular forms and affine Lie algebras.
\end{abstract}

\bigskip

\section{The Borcherds modular form $\Phi_{12}$ and Lorentzian  and affine Kac--Moody algebras}
\bigskip
\thispagestyle{empty}

In 1983, I. Frenkel and A. Feigold posed a question about 
possible relations between the simplest hyperbolic Kac--Moody algebra,
Siegel modular forms and affine Lie algebras (see \cite{FF}).
A Siegel modular form of genus $2$ can be considered as 
a  modular form on the  orthogonal 
group $\Orth(2,3)$. The Weyl--Kac denominator functions of affine 
Lie algebras are Jacobi modular forms similar to Fourier--Jacobi coefficients 
of  modular forms on $\Orth(2,n)$.

In \cite{GN2} we constructed {\it the automorphic correction} of the simplest 
hyperbolic Kac--Moody algebra, i.e. a generalized hyperbolic Kac--Moody super Lie algebra
with the same (hyperbolic) real simple roots 
whose Weyl--Kac--Borcherds denominator function 
is the classical Igusa cusp form $\Psi_{35}$ of weight $35$.
($\Psi_{35}$ is essentially the unique Siegel modular form of odd weight.) 
The multiplicities of positive roots of this Lorentzian Kac--Moody algebra 
are given by the Fourier coefficients  
of a  weakly holomorphic Jacobi form of weight $0$.
The first non-zero Fourier--Jacobi coefficient of $\Psi_{35}$
is equal to the Jacobi modular form $\eta(\tau)^{69}\vartheta(\tau, 2z)$
of weight $35$ and index $2$ which is not a denominator function
of affine type.
In the class of Lorentzian Kac--Moody algebras of hyperbolic rank $3$
classified in \cite{GN3}--\cite{GN5}
there exists an  algebra determined by the even  Siegel  theta-series 
$\Delta_{1/2}$ whose first Fourier--Jacobi coefficient is equal to 
$\vartheta(\tau, z)$ which is 
the denominator function of the simplest affine Lie algebra $\hat \mg(A_1)$.

In this paper we analyze  the fake monster Lie algebra discovered  
by R.~Borcherds in
1995 (see \cite{B1}). The algebra is determined by the Borcherds
modular form $\Phi_{12}$ of weight $12$ with respect to $\Orth^+(II_{2,26})$
where $II_{2,26}$ is the even integral lattice of signature $(2,26)$.
The modular form $\Phi_{12}$ has $24$ different Fourier--Jacobi expansions
in the $24$ one-dimensional cusps ($23$ Niemeier cusps and the Leech cusp).
We prove that $\Phi_{12}$ vanishes at all $23$ Niemeier cusps
and  its   first non-zero  Fourier--Jacobi coefficient 
is equal to the Weyl--Kac denominator function 
of the affine Lie algebra of the root system of the corresponding Niemeier lattice.
This result  is an automorphic answer
on the question of Frenkel and Feingold  in the case of the fake monster Lie algebra.

Lorentzian Kac--Moody algebras are hyperbolic analogue
of the finite dimensional and affine Lie algebras
(see \cite{B1}--\cite{B2} and \cite{GN1}--\cite{GN5}).
A Lorentzian Kac--Moody algebra is graded  by a hyperbolic root lattice $L_1$,
its Weyl group is a hyperbolic  reflection group 
and its Weyl--Kac--Borcherds denominator function is 
a modular form with respect to an orthogonal group  of signature
$(2,n)$. 

The famous example is the {\it fake monster Lie algebra} $\gG_\Lambda$
with  the root lattice $II_{1,25}\cong U\oplus \Lambda_{24}(-1)$  
of signature $(1,25)$, where 
$U\cong
\begin{pmatrix}
0&1\\1&0
\end{pmatrix}$  is the hyperbolic plane 
and $\Lambda_{24}(-1)$ is the rescaled Leech lattice,
the negative definite even unimodular lattice of rank $24$ without vectors
of square $2$. 
One  can consider the Borcherds modular form 
$$
\Phi_{12} \in M_{12}(\Orth^+(II_{2,26}),\,\det)
$$
as the generating function of $\gG_\Lambda$ because
it contains the full information on the generators and relations of the algebra
and on the  multiplicities of all positive roots. 
The divisor of $\Phi_{12}$ is the union of all rational quadratic divisors
determined by $-2$-vectors in $II_{2,26}$.

Let $L_2$ be an even integral   lattice with a
quadratic form of signature $(2,n)$,
$$
\cD(L_2)=\{[\cZ] \in \PP(L_2\otimes \CC) \mid
  (\cZ,\cZ)=0,\ (\cZ,\overline{\cZ})>0\}^+
$$
be the associated $n$-dimensional bounded symmetric  Hermitian domain  of type $IV$
(here $+$ denotes one of its two connected components).
We denote by $\Orth^+(L_2)$ the index $2$ subgroup of the integral orthogonal group 
$\Orth(L_2)$ preserving $\cD(L_2)$.
The domain contains 
the following   {\it rational quadratic divisors}
$$
\cD_v=\{[\cZ] \in \cD(L_2)\mid (\cZ,v)=0\}\cong \cD(v^\perp_{L_2})
 \quad {\rm where }\quad
v\in L_2^\vee,\ (v,v)<0
$$
and $L_2^\vee$ is the dual lattice.
The modular quotient  $\Gamma\setminus \cD(L_2)$ where $\Gamma$ is a subgroup of finite
index in $\Orth^+(L_2)$ is a quasi-projective variety. Its Baily--Borel compactification
contains only boundary components of dimension $0$ and $1$ (see \cite{BB}).

\begin{lemma}\label{lem-BBC}
The Baily--Borel compactification of the  quasi-projective modular variety  
 $\Orth^+(II_{2,26})\setminus \cD(II_{2,26})$ 
 is a bouquet of $24$ modular curves $\SL_2(\ZZ)\setminus \HH$ 
 with the common zero-dimensional cusp.
\end{lemma}
\begin{proof}
According to \cite{BB} the zero- and one-dimensional boundary components
of $\Orth^+(II_{2,26})\setminus \cD(II_{2,26})$ 
correspond to the $\Orth^+(II_{2,26})$-orbits of the primitive 
isotropic vectors and totally isotropic planes respectively.
There exists only one orbit of primitive isotropic vectors in $II_{2,26}$,
i.e. only one zero-dimensional cusp.
Let $F<II_{2,26}$ be a primitive totally isotropic sublattice of rank $2$.
Then $L(-1)=F^\perp/F$ is a negative even unimodular lattice 
(see \cite[\S 2.3]{GHS1} for more details).
There are exactly $24$ classes of such lattices. They are 
the $23$ Niemeier lattices $N(R)$ uniquely determined by its root lattice $R$
of rank $24$
\begin{gather*}
3E_8,\ E_8\oplus D_{16},\  D_{24},\  2D_{12},\  3D_8,\ 
4D_6,\ 6D_4,\\
A_{24},\  2A_{12},\  3A_8,\  4A_6,\  6A_4,\  8A_3,\  12A_2,\  24A_1,\\
E_7\oplus A_{17},\ 2E_7\oplus D_{10}, \ 4E_6,\ E_6\oplus D_{7}\oplus A_{11},\\
A_{15}\oplus D_{9},\ 2A_{9}\oplus D_{6},\ 2A_{7}\oplus D_{5},\ 4A_{5}\oplus D_{4}
\end{gather*}
and the Leech lattice $\Lambda_{24}$ without roots
(see \cite[Chapter 18]{CS}). 
In other words, there are 
$24$ different models of the unique (up to isomorphism) even unimodular lattice
$II_{2,26}$  of signature $(2,26)$
corresponding to the  
 one-dimensional cusps
$$
II_{2,26}\cong U\oplus U\oplus N(-1)
$$
where $N(-1)$ is  the  negative definite Niemeier lattice $N=N(R)$.
(The bilinear form on $N(-1)$ is equal to  $-( .\,,\,. )_{N}$.)
The corresponding one dimensional boundary 
components 
are the modular curves $\SL_2(\ZZ)\setminus \HH$. All of them have the common
zero-dimensional cusp.
\end{proof}

In this paper we analyze the Fourier-Jacobi expansions of $\Phi_{12}$.
For this end we define   the  tube domain realizations  of $\cD(II_{2,26})$
corresponding to one-dimensional cusps. We construct  it in more general context.
Let 
$$
L_2=U\oplus U_1\oplus L(-1)
$$
where $U\cong U_1$ are two hyperbolic planes and $L$ is an even integral
positive definite lattice.
We fix a basis of the first hyperbolic plane $U=\ZZ e\oplus\ZZ f$:
$(e,f)=1$ and $e^2=f^2=0$ in $L_2$. Similarly  $U_1\cong \ZZ e_1\oplus\ZZ f_1$.
We choose a basis of $L_2$ of the form
$(e,e_1,\dots , f_1,f)$ where $\dots$ denote a basis of $L(-1)$.
The one dimensional boundary component related to $L(-1)$ is defined 
by the isotropic tower $\latt{f}\subset \latt{f,f_1}$.
Let $[\cZ]=[\cX+i\cY] \in \cD(L_2)$. Then
$(\cX,\cY)=0$, $(\cX,\cX)=(\cY,\cY)$ and $(\cZ,\overline{\cZ})=2(\cY,\cY)>0$.
Using the basis $\latt{e,f}_\ZZ=U$ we write $\cZ=z'e+\widetilde Z+zf$ with
$\widetilde Z\in L_1\otimes \CC$ where
$L_1=U_1\oplus L(-1)$ is a hyperbolic sublattice  of $L_2$. 
We note that $z\ne 0$. (If $z=0$, then  the real and imaginary parts of
$\widetilde Z$
are two orthogonal vectors of positive norm in the real hyperbolic space
$L_1\otimes \RR$ of signature $(1,n+1)$.)
Thus $[{}^t\cZ]=[(-\frac{1}2(Z,Z)_1, {}^tZ,1)]$
where $(.,.)_1$ is the hyperbolic  bilinear form in $L_1$.
 Using the basis $\latt{e_1,f_1}_\ZZ=U_1$
of the second hyperbolic plane  we see that
$\cD(L_2)$ is isomorphic to the tube domain
\begin{equation*}\label{HL}
\cH(L)=
\{Z=\left(\smallmatrix\omega\\ \mathfrak{Z}\\ \tau
\endsmallmatrix\right),\
\tau,\,\omega\in\HH,\  \gz\in L\otimes \CC,\,
(\Im Z, \Im Z)_1>0\}
\end{equation*}
where $(\Im Z, \Im Z)_1=2\Im(\omega)\Im(\tau)-(\Im(\gz),\Im(\gz))_L$.
We note that $\mathcal{H}(L)$ is the complexification 
of the connected light cone $V^+(L_1)=\{Y\in L_1\otimes \RR \ |\ (Y,Y)>0\}^+$.
We fix the isomorphism $[{\rm pr}]:\cH(L)\to\cD(L_2)$
defined
by the  $1$-dimensional  cusp $L$ fixed above
$$
Z=\left(\begin{smallmatrix}\omega\\ \gz \\ \tau\end{smallmatrix}\right)\
{\mapsto}\ {{\rm{pr}}(Z)}=
\left(\begin{smallmatrix}-\frac{1}{2}(Z,Z)_1\\
\omega\\ \gz \\
\tau\\1\end{smallmatrix}\right)
\mapsto\ \left[{\rm{pr}}(Z)\right]\in \cD(L_2).
$$
\smallskip

The root lattice of the fake monster Lie algebra $\gG_\Lambda$ 
is the hyperbolic lattice $\Lambda_{1,25}=U\oplus \Lambda_{24}(-1)$ 
where $\Lambda_{24}$
is the Leech lattice.
(Since we are working with modular forms we inverse the signature
of the root lattice.)
The Weyl group $W=W_{-2}(\Lambda_{1,25})$ of $\gG_\Lambda$ is generated by 
the reflections in all elements $v\in \Lambda_{1,25}$ with $v^2=-2$.
The group $W$ is discrete in the hyperbolic space
$\cL (\Lambda_{1,25})=V^+(\Lambda_{1,25})/\RR_{>0}$.
The (infinite) set $P$ of the real simple roots of $\gG_\Lambda$
contains the $-2$-vectors which 
are orthogonal to the walls of a fundamental chamber of $W$ in $\cL (S)$.
The set $P$ has a Weyl vector $\rho$. 
For example, one can take $\rho=e_1$ where $e_1$ is the first isotropic vector
of the basis for $U_1$. Then 
$
P=\{v \in \Lambda_{1,25}\ |\ (v,\,v)=-2\ {\rm and}\ (\rho,v)=1 \}
$. 

In \cite{B1} the Borcherds form $\Phi_{12}$ was determined in 
the Leech cusp of the modular group $\Orth^+(II_{2,26})$.
In order to describe the Fourier expansion of $\Phi_{12}$
(one has  only one Fourier expansion)
we  need the Ramanujan cusp form  $\Delta$ of the weight $12$
$$
\Delta(\tau)=q\prod_{n=1}^{\infty}(1-q^n)^{24}=
\sum_{m\ge 0}\tau(m)q^m,\qquad 
\Delta^{-1}(\tau)=\sum_{n\ge 0}p_{24}(n)q^{n-1}.
$$
Then one has   the following two expressions for 
the Borcherds modular form  (see \cite[\S 10]{B1}) 
$$
\Phi_{12}(Z)=\exp{(2\pi i (\rho,Z))}\prod_{\alpha\in \Delta_+}
{(1-\exp{(2\pi i (\alpha,Z))})^{p_{24}(1-(\alpha,\alpha)/2)}}=
$$
\begin{equation}\label{eq-PhiFE}
\sum_{w\in W}{\det(w)\sum_{m>0}{\tau (m)\exp{(2\pi i (w(m\rho),Z))}}}
\end{equation}
where $Z\in \cH(\Lambda_{24})$ and 
$\Delta_+=\{\alpha \in \Lambda_{1,25}\,|\,\alpha^2=-2\ {\rm and}
\ (\alpha,\,\rho)>0\}\cup
(\Lambda_{1,25}\cap \overline{V^+(\Lambda_{1,25})}-\{0\})$ 
is the set of positive roots 
of $\gG_\Lambda$.
The last identity between the sum over $W$ and the infinite product 
over $\Delta_+$  is {\it the Weyl--Kac--Borcherds identity}
for the fake monster Lie algebra $\gG_\Lambda$.

We note that the  Weyl--Kac denominator function of an affine Lie algebra
is a holomorphic  Jacobi modular  form (see \cite{Ka}, \cite{KP}).
The denominator identity for the simplest affine Lie 
algebra $\hat\mg(A_1)$ is reduced to the Jacobi triple
product identity
\begin{equation}\label{theta}
\vartheta(\tau,z)=\sum_{n\in\ZZ}\left(\frac{-4}{n}\right)q^{\frac{n^2}{8}}
r^{\frac{n}{2}}
=-q^{1/8}r^{-1/2}\prod_{n\geqslant 1}\,(1-q^{n-1} r)(1-q^n r^{-1})(1-q^n)
\end{equation}
where $q=e^{2\pi i \tau}$, $r=e^{2\pi i z}$, $z\in \CC$ and 
$\left(\frac{-4}{n}\right)$ is the Kronecker symbol.
The last  function is the  Jacobi theta-series
of characteristic $(\frac{1}2, \frac{1}2)$.
A question about relations
between the affine Lie algebras and the simplest hyperbolic Lie
algebra  with the Cartan matrix
$\begin{pmatrix} \ \,\,2&-2&\ \,\,0\\-2&\ \,\,2&-1\\\ \,\,0&-1&\ \,\,2\end{pmatrix}$ 
was posed  in 1983 in the paper 
of I.~Frenkel and A.~Feingold \cite{FF}. 
Theorem \ref {thm-PhiAA} below
is an automorphic answer
on a similar   question in the case of the fake monster Lie algebra.
\smallskip

The Fourier expansion (\ref{eq-PhiFE})
shows that the value of $\Phi_{12}(Z)$ at the one dimensional cusp
determined by the Leech lattice is equal to $\Delta(\tau)$
(see also the Fourier-Jacobi coefficient of index one 
in (\ref{BLe}) below).
The root lattice of $\Lambda_{24}$ is trivial.
A Niemeier lattice $N(R)$ is  uniquely determined by its root system $R$.
The list of possible $R$ was given in the proof of Lemma \ref{lem-BBC}.
We note that if $R=R_1\oplus\dots \oplus R_m$ is reducible, 
then the Coxeter numbers $h(R_i)$ of all components of the  Niemeier root system 
$R$ are the same.
We denote this number by $h(R)$.
\begin{theorem}\label{thm-PhiAA}
Let $N(R)$ be a Niemeier lattice with a non-empty root system $R$.
The Borcherds modular form $\Phi_{12}$ vanishes with order $h(R)$ along the 
one-dimensional boundary component determined by $N(R)$. 
The first non-zero Fourier-Jacobi coefficient of $\Phi_{12}$  in this cusp
is, up to a sign, the Weyl--Kac denominator function of the affine Lie algebra 
$\hat{\mathfrak g}(R)$
$$
\Phi_{12}(\tau, \gz, \omega)=\pm \eta(\tau)^{24}\prod_{v\in R_+}
\frac{\theta(\tau, (v,\gz))}{\eta(\tau)}\,e^{2\pi i h(R)\omega}+\dots
$$
where the product is taken over all positive roots of the finite root system $R$
of rank $24$. The sign in the formula depends on the choice of the positive roots  
in $R$. 
\end{theorem}
We prove this theorem in \S 3 where we give explicit formulae for the first 
three Fourier--Jacobi coefficients at all one dimensional cusps
including the Leech cusp. 
\smallskip

\section{Jacobi modular forms in many variables}
\medskip

In this section we discuss Jacobi modular forms of orthogonal type.
In the definitions we follow \cite{G1}--\cite{G2}
where Jacobi forms were considered as modular forms with respect to
a parabolic subgroup of an orthogonal group of signature $(2,n)$.
We mentioned in \S 1 that  a one-dimensional cusp of $\cD(L_2)$
is defined by a maximal totally isotropic tower $\latt{f}\subset \latt{f,f_1}=F<L_2$.
A Jacobi modular form in our approach is a modular form with respect 
 to the integral parabolic subgroup $P_F<\Orth^+(L_2)$ 
(see \cite{G2} and \cite{CG2} for more details).
The regular part of  $P_F$ is  the so-called {\it Jacobi modular group} 
$\Gamma^J(L)\cong \SL_2(\ZZ)\rtimes H(L)$ where $\SL_2(\ZZ)$ acts on the isotropic plane 
$F$ and  $H(L)$ is the Heisenberg group 
acting trivially on the totally isotropic plane $F$.
The group $H(L)$ is a central extension of $L\times L$ and any element $h\in H(L)$ 
can be written in the form
$h=[\lambda,\mu;\kappa]$ where $\lambda$, $\mu\in L$, 
$\kappa\in \frac{1}{2}\ZZ$ and $\kappa+\frac{1}{2}(\lambda,\mu)\in \ZZ$.
We note that $P_F$  is the product of the Jacobi group and the finite
orthogonal group of the positive definite lattice $L$.
Analyzing the holomorphic function $\varphi(\tau, \gz)e^{2\pi i m \omega}$ 
on $\cH(L)$  
modular with respect to  the  semi-simple part 
$\SL_2(\ZZ)$ and the unipotent part $H(L)$ 
of  the Jacobi group we obtain the following definition.
\begin{definition}
A holomorphic (resp. cusp or  weak) Jacobi form of weight $k\in \ZZ$ 
and index $m\in \NN$
for  $L$ is  a holomorphic function
$$
\phi: \HH\times (L\otimes \Bbb C)\to \CC
$$
satisfying the functional equations
\begin{align*}\label{def-JF}
\phi(\frac{a\tau+b}{c\tau+d},\,\frac{\gz}{c\tau+d})&
=(c\tau+d)^{k}\exp\bigl (\pi i \frac{cm(\gz,\gz)}{c\tau+d}\bigr)
\phi(\tau ,\,\gz ),
\\\vspace{2\jot}
\phi(\tau,\gz+\lambda\tau+\mu)&
=\exp\bigl(-\pi i m\bigl( (\lambda,\lambda)\tau+ 2(\lambda,\gz)\bigr)\bigr)
\phi(\tau,\gz )
\end{align*}
for any
$A=\begin{pmatrix}
a&b\\c&d
\end{pmatrix}\in{\SL}_2(\Bbb Z)$ and any
$\lambda,\,\mu\in L$
and having a Fourier expansion
$$
\phi(\tau ,\,\gz )=
\sum_{n\in \ZZ,\ \ell \in L^\vee}
f(n,\ell)\,\exp\bigl(2\pi i (n\tau+(\ell,\gz) \bigr),
$$
where $n\ge 0$ for a {\bf weak} Jacobi form,
$N_m(n,\ell)=2nm-(\ell,\ell)\ge 0$ for a {\bf holomorphic} Jacobi form
and $N_m(n,\ell)>0$ for a {\bf cusp} form. 
\end{definition}

We denote the space of all holomorphic Jacobi forms by
$J_{k,m}(L)$. We use the  notation $J_{k,m}^c(L)$, $J_{k,m}^w(L)$
and $J_{k,m}^{wh}(L)$
for the space of cusp, weak or weakly holomorphic  Jacobi forms.
$\varphi$ is called {\it weakly holomorphic} if there exists $N$ such that 
$\Delta^N(\tau)\varphi(\tau,\gz)\in J_{k,m}^w(L)$. 

If $J_{k,m}(L)\ne \{0\}$, then $k\ge \frac{1}2\rank L$ (see \cite{G1}).
The weight $k=\frac{1}2\rank L$ is called {\it singular}.
The denominator function of an affine Lie algebra is a holomorphic Jacobi 
form of singular weight (see \cite{Ka} and \cite{KP}).

It is known (see \cite[Lemma 2.1]{G2}) that
$f(n,\ell)$ depends only on the hyperbolic norm
$N_m(n,\ell)=2nm-(\ell,\ell)$ and the image of $\ell$
in the discriminant group $D(L(m))=L^\vee/mL$
($L(m)$ denotes the rescaling of the  lattice $L$ by $m$).
Moreover,
$f(n,\ell)=(-1)^kf(n,-\ell)$. 
We note that  $J_{k,m}(L)=J_{k,1}(L(m))$
and the space $J_{k,m}(L)$ depends essentially only on the discriminant form
of $L(m)$ (see \cite[Lemma 2.4]{G2}).
\smallskip

{\bf Example 1.} {\it Jacobi theta-series of singular weight} 
(see \cite{G1}--\cite{G2}).
Let $L$ be an even unimodular lattice of rank $n\equiv 0 \mod 8$. Then 
\begin{equation*}\label{eq-Jth}
\vartheta_{L}(\tau, \mathfrak{z})=
\sum_{\ell\in L}\exp{\bigl(\pi i (\ell,\ell)\tau+2\pi i (\ell, \mathfrak{z} )\bigr)}
\in J_{\frac{n}2,1}(L).
\end{equation*}

One can also  define Jacobi forms of integral or half-integral weights 
with a character (or multiplier system) of  fractional  index
(see \cite{GN4}, \cite{CG2}). 
 
A Jacobi form determines a vector valued modular form
related to the corresponding Weil representation (see \cite[Lemma 2.3]{G2}).
In particular, for $\varphi\in J_{k,1}(L)$ we have
\begin{equation}\label{vvmf}
\varphi(\tau ,\,\gz)=\sum_{\substack{  n\in \ZZ,\ \ell \in L^\vee
\vspace{0.5\jot} \\
 2n-(\ell,\ell)\ge 0}}
f(n,\ell)\,\exp\bigl(2\pi i (n\tau+(\ell,\gz) \bigr)
=\sum_{h\in D(L) } \phi_h(\tau)\Theta_{L,h}(\tau, \gz),
\end{equation}
where $h\in D(L)=L^\vee/L$, $\Theta_{L,h}(\tau, \gz)$ is the Jacobi theta-series with characteristic
$h$ and the components of the vector valued modular forms
$(\phi_h)_{h\in D(L)}$ have the following Fourier expansions at infinity
$$
\phi_h(\tau)=\sum_{\substack{r\equiv -\frac{1}{2}(h,h)\mod \ZZ}}
f_h(r)
\exp{(2\pi i\, r\tau)}
$$
with
$f_h(r)=f(r+\frac{1}{2}(h,h),h)$. 
We note that a vector valued modular form
depends on the genus of the lattice $L$, i.e. on the discriminant group $D(L)$.
A Jacobi form contains
information on the class of $L$. We shall use this property
in order to prove Theorem 1.2. Moreover the Jacobi forms  
build a natural  bigraded ring  with respect to weights and indices. 
This fact is  useful for many constructions (see \cite{G3} and \cite{GN4}).
\smallskip

{\bf Example 2.}
If $L=A_1=\langle2\rangle$, then $J_{k,m}(A_1)=J_{k,m}$
is the space of holomorphic Jacobi modular forms of Eichler--Zagier type
studied in the book \cite{EZ}. The isomorphism 
$J_{k,m}\cong J_{k,1}(\langle -2m\rangle)$ was one of the main starting points
for the construction of the Jacobi lifting for the paramodular Siegel groups
in \cite{G2}. 
\smallskip

In the context of Borcherds products, 
reflective Siegel  modular forms and
the corresponding Lorentzian Kac-Moody algebras
(see \cite{GN1}--\cite{GN5}), the basis Jacobi form 
is the Jacobi theta-series 
$$
\vartheta(\tau,z)\in J_{\frac{1}2, \frac{1}2}(v^3_\eta\times v_H)
\qquad ({\rm see\ }\  (\ref{theta}))
$$
which is the  Jacobi form of weight $\frac{1}2$ and index $\frac 1{2}$
with multiplier system $v_\eta^3\times v_H$,
where $v_\eta$ is the multiplier
system of order $24$ of the Dedekind eta-function $\eta$
and  $v_H$ is the binary character of the Heisenberg group
(see \cite[Example 1.5]{GN4}).
In particular,  we have 
$$
\vartheta\bigl(\frac{a\tau+b}{c\tau+d},\,\frac{z}{c\tau+d}\bigr)
=v_\eta^3(M)(c\tau+d)^{1/2}\exp(-\pi i \frac{cz^2}{c\tau+d})
\vartheta(\tau ,\,z )
$$
for all $M=\left(\begin{smallmatrix}
 a&b\\c&d
\end{smallmatrix}\right) \in \SL_2(\ZZ)$  and 
$$
\vartheta(\tau ,z+\lambda \tau +\mu)=
(-1)^{\lambda+\mu}\exp{(-\pi i\,(\lambda^2 \tau  +2\lambda z))}\,
\vartheta(\tau , z)\quad (\lambda,\mu \in \ZZ).
$$
The next proposition shows
the role of $\vartheta(\tau ,\,z)$ in the context of this paper.
\begin{proposition}\label{prop-JObst}
Let 
$$
\varphi(\tau,\gz)=\sum_{n\in \ZZ,\,\ell\in L^\vee} f(n,\ell)q^nr^\ell
\in J_{0,1}^{wh}(L)
$$
be a weakly holomorphic Jacobi form of weight $0$ and index $1$ for the lattice $L$.
The Fourier coefficients $f(0,\ell)$ determine a generalized  $2$-designe in $L^\vee$.
More exactly the following identity is valid
\begin{equation}\label{eq-2des}
\sum_{\ell\in L^\vee} f(0,\ell)(\ell, \gz)^2=2C(\gz,\gz),
\qquad \forall\,\gz\in L\otimes \CC
\end{equation}
where 
$$
C=\frac{1}{24}\sum_{\ell\in L^\vee}f(0,\ell)
-\sum_{n>0,\,\ell\in L^\vee} f(-n,\ell)\sigma_1(n)
=\frac{1}{2\rank L}\sum_{\ell\in L^\vee} f(0,\ell)(\ell,\ell).
$$
\end{proposition}
\begin{proof}
We prove the proposition using the method of automorphic correction
of Jacobi forms
(see \cite[Lemma 1.10]{G3}, \cite[Proposition 1.5]{G4} and  \cite[(6)]{CG1}).
We consider the quasi-modular Eisenstein series of weight $2$
$$
G_2(\tau)=-D(\log (\eta(\tau)))=-\frac{1}{24}+\sum_{n\ge 1}\sigma _1(n)q^n,
\quad q=e^{2\pi i \tau},
\ 
D=q\frac{d\ }{dq}
$$
where $\sigma_k(m)=\sum_{d|m} d^k$.
We define {\it the automorphic correction} of the weakly holomorphic Jacobi form
$\varphi$ as follows
$
\varphi_{cor}(\tau,\gz)=e^{-4\pi ^2G_2(\tau)(\gz,\gz)}\varphi(\tau,\gz)
$. 
This function transforms like an automorphic function in $\tau$
$$
\varphi_{cor}(\frac{a\tau+b}{c\tau+d},\frac{\gz}{c\tau+d})=\varphi_{cor}(\tau,\gz).
$$
Therefore  the sum of the coefficients of order $2$ in the Taylor expansion 
of $\varphi_{cor}(\tau,\gz)$ in $\gz$ is   
a meromorphic $\SL_2(\ZZ)$-modular  form of weight $2$ in $\tau$.
It is easy to find its zeroth Fourier coefficient. It is equal to 
$$
(2\pi i)^2
\biggl(\sum_{\ell\in L^\vee}
f(0,\ell)(\ell,\gz)^2+2(\gz,\gz)\bigl[-\frac{1}{24}\sum_{\ell\in L^\vee}f(0,\ell)
+\sum_{n>0,\,\ell\in L^\vee} \sigma_1(n)f(-n,\ell)\bigr]
\biggr).
$$
The differential operator  $D$ maps   the space of modular  functions of weight 
$0$ onto  the space of meromorphic modular   forms of weight $2$. In particular 
the zeroth Fourier coefficient of a meromorphic $\SL_2(\ZZ)$-modular  form of weight $2$
is equal to zero.  It proves the identity (\ref{eq-2des}) 
with the first expression for $C$.
After that we can also find the second expression for $C$ 
acting by  the  Laplace operator 
on (\ref{eq-2des}) because 
$\Delta_\gz(\ell,\gz)^m=m(m-1)(\ell,\ell)(\ell,\gz)^{m-2}$
and $\Delta_\gz(\gz,\gz)=2\rank L$. 
\end{proof}

A finite multiset of  vectors $\{(\ell; m(\ell))\}$ from $L^\vee$
(one takes every vector $m(\ell)$ times) which satisfies 
(\ref{eq-2des}) is called {\it vector system} in \cite[\S 6]{B1}.
One can introduce positive  and negative elements ($\ell>0$ and $\ell<0$) 
in a vector system using the sign of the scalar product with a vector 
which is not orthogonal to any vectors in the system.
Then every element will be  either positive or negative.
If $L$ is even integral, 
then $s(L)$ (respectively, $n(L)$) denotes 
the positive generator of the integral ideal 
generated by $(\lambda,\mu)$  (respectively, by  $(\lambda,\lambda)$)
for $\lambda$ and $\mu$ in $L$.
 
\begin{corollary}\label{cor-Ptheta}
Let $\varphi$ be as in Proposition \ref{prop-JObst}.
We assume that all Fourier coefficients $f(0,\ell)$ are integral.
Then the function
$$
\psi_\varphi(\tau,\gz)=\prod_{\ell\in  L^\vee,\,\ell > 0}
\biggl(\frac{\vartheta(\tau,(\ell, \gz))}{\eta(\tau)}\biggr)^{f(0,\ell)}
$$
where the product is taken with respect to a fixed ordering in $L$
mentioned above,
transforms like Jacobi form of weight $0$ and index $C$ 
$$
C=\frac{1}{2\rank L}\sum_{\ell\in L^\vee} f(0,\ell)(\ell,\ell)
$$
for $L$ with a character. Let 
$\vec B=\frac 1{2}\sum_{\ell\in L^\vee,\,\ell>0}f(0,\ell)\ell\in \frac{1}2L^\vee$.
If 
$C\cdot n(L)\in 2\ZZ$, then  $C\cdot s(L)\in \ZZ$ and 
$\vec B\in L^\vee$.
\end{corollary}
\begin{proof}
Using the functional equations of $\vartheta(\tau,z)$
we  obtain that the theta-product 
transforms like a Jacobi form  of weight $0$ and  index $C$
with a character 
if an only if  the system $\{(\ell, f(0,\ell))\}$ 
satisfies  (\ref{eq-2des}).
This is clear for $\SL_2(\ZZ)$-transformations.
In order to prove the same for the abelian translations
$\gz\to \gz+\lambda \tau +\mu$ one uses the bilinear variant of (\ref{eq-2des})
$$
\sum_{\ell\in L^\vee,\, \ell>0}
f(0,\ell)(\ell, \gz_1)(\ell,\gz_2)=C(\gz_1,\gz_2),
\qquad \forall\,\gz_1,\gz_1\in L\otimes \CC.
$$
For any $\lambda$, $\mu\in L$ we have 
$$
\psi_\varphi(\tau,\gz+\lambda \tau +\mu)=
(-1)^{\sum_{\ell>0} f(0,\ell)(\ell,\lambda+\mu)}
e^{-\pi i C( (\lambda,\lambda)\tau+ 2(\lambda,\gz))}
\psi_\varphi(\tau,\gz).
$$
Therefore the restriction of the Jacobi character $\chi$ of  $\psi_\varphi$
to  the minimal Heisenberg subgroup
$H_s(L)$  generated by the elements in $L\times L$ is a binary character.
In fact one can prove (see \cite[\S 1]{CG2}) that
$$ 
\chi|_{H(L)}([\lambda,\mu;\kappa])=
e^{\pi i C((\lambda,\lambda)+(\mu,\mu)-(\lambda,\mu)+2\kappa)},
\quad [\lambda, \mu; \kappa] \in H(L).
$$
The properties of  the index $C$ follow from \cite[Proposition 1.1]{CG2}.
The property of $\vec B$ reflects the fact that the binary character 
of $H_s(L)$ is trivial, if $C\cdot n(L)\in 2\ZZ$.
The $\SL_2$-part of the Jacobi character is equal 
to $v_\eta^{2A'}$ where
$A'=\sum_{\ell\in L^\vee , \,\ell>0} f(0,\ell)$.

\end{proof}

\noindent
{\bf Remark.} 
The theta-product of the corollary is equal to the product function $\psi$
on the page 183 of \cite{B1}. In particular, 
Corollary \ref{cor-Ptheta} gives another simple  proof 
of  \cite[Theorem 6.5]{B1}. Our proof does not use the Poisson summation formula.

\section{Borcherds products and Jacobi forms}

In this section, we write Borcherds automorphic products
on $\Orth^+(L_2)$ 
in terms of Jacobi modular forms  of weight $0$
for the even  positive definite  lattice $L$.
In \cite{B2}, the language of   vector valued automorphic forms
was used.  The main  theorem of this section  is a natural  generalization 
of \cite[Theorem 2.1]{GN4} where 
Siegel modular forms with respect to  the symplectic paramodular group $\Gamma_t$
were considered.  This subject is very   natural. It was given in my
lecture course for Ph.D. students in Lille in 2002/03 without publishing 
a separate paper (see the chapter 4 in the  dissertation of C.~Desreumaux (2004)
where the corresponding formulations were used).
New applications of Borcherds products 
in algebraic geometry (see \cite{GHS1}--\cite{GHS2}, \cite{G5} and \cite{GH}),
in string theory (see \cite{CG1} and the references there)
and to the classification theory of Lorentzian Kac--Moody algebras
make this subject  actual again. 

We recall the definition of the stable orthogonal group
$$
\widetilde{\Orth}^+(L)=\{g\in \Orth^+(L)\,|\,g|_{L^\vee/L}=\id\}.
$$  

\begin {theorem}\label{product}
Let
$$
\varphi(\tau,\gz)=\sum_{n\in \ZZ,\,\ell\in L^\vee} f(n,\ell)q^nr^\ell
\in J_{0,1}^{wh}(L)
$$
be a weakly holomorphic Jacobi form of weight $0$
and index $1$ for an even integral positive definite lattice $L$.
We fix an ordering ($>0$ and $<0$) in the lattice $L$ like in 
Corollary \ref{cor-Ptheta}.
Assume that 
$f(n,\ell)\in \ZZ$ if  $2n-\ell^2\le 0$.   
Then, the product
$$
\cB_\varphi(Z)=q^{A}r^{\vec B}s^{C}
\prod_{\substack{n,m\in \ZZ,\,\ell\in L^\vee \vspace {1pt} \\ (n,\ell,m)>0}}
\bigl(1-q^nr^\ell s^{m}\bigr)^{f(nm,\ell)},
$$
where $Z=(\tau, \gz,\omega)\in \cH(L)$,
$q=\exp{(2\pi i \tau)}$, $r^\ell=\exp{(2\pi i (\ell, \gz))}$, 
$s=\exp{(2\pi i \omega)}$,
$(n,\ell,m)>0$ means  that $m>0$, or 
$m=0$ and $n>0$, or $m=n=0$ and  $\ell<0$, 
$$
A=\frac{1}{24}\sum_{\ell\in L^\vee} f(0,\ell),\ \
\vec B=\frac{1}{2}\sum_{\ell >0 }
f(0,\ell)\ell\in \frac{1}2L^\vee,\ \ 
C=\frac{1}{2\rank L}\sum_{\ell \in L^\vee}
f(0,\ell)(\ell,\ell)
$$
defines a meromorphic modular  form of weight
$
k=\frac{1}{2}f(0,0)
$
with respect to $\widetilde\Orth^+(L_2)$ with a character
$\chi$ (see (\ref{eq-char}) below). 
The poles and zeros of $\cB_\varphi$ lie on the rational
quadratic divisors defined by the Fourier coefficients
$f(n,\ell)$ with   $2n-\ell^2< 0$.
In particular, $\cB_\varphi$ is holomorphic if all  such coefficients are positive.
The explicit formula for the multiplicities is given in (\ref{eq-mult}).
\end{theorem}
\begin{proof}

The product of the theorem is a special case  of the Borcherds automorphic products
considered in \cite[Theorem 13.3]{B2} because the Jacobi form $\varphi$
determines a vector valued modular form of weight $-\frac{\rank L}2$ 
according (\ref{vvmf}).
The product converges if $Y=\Im Z$ with 
$(Y, Y)>M$ for a sufficiently large $M$ lies in a fundamental 
domain of the hyperbolic  orthogonal group $\widetilde\Orth^+(L_1)$ 
acting on the cone $V^+(L_1)$.

We define the invariants of the automorphic products
(the modular  group, the weight, the character,
the multiplicities of the poles and the zeros, 
the first several members of the  Fourier-Jacobi expansions)
in terms of the Fourier coefficients of the Jacobi form
$\varphi$ using a Hecke type representation of the Borcherds products 
given in  \cite{GN4}.

The choice of the vector  $(A, \vec B, C)$
in the definition 
of $\cB_\varphi$ is motivated  by the following identity
$$
q^{A}r^{\vec B}s^{C}
\prod_{(n,\ell,0)>0}
(1-q^nr^\ell)^{f(0,\ell)}=
\eta(\tau)^{f(0,0)}
\prod_{\ell>0}
\biggl(\frac{\vartheta(\tau,(\ell, \gz))e^{ \frac {2\pi i}{\rank L}(\ell,\ell)\omega}}
{\eta(\tau)}\biggr)^{f(0,\ell)}.
$$
The vector $(A, \vec B, C)$ is called usually 
{\it Weyl vector} of the Borcherds product.
According to Proposition \ref{prop-JObst} and Corollary \ref{cor-Ptheta}
$$
\Tilde\psi_{L;C}(Z)=
\eta(\tau)^{f(0,0)}
\prod_{\ell>0}
\biggl(\frac{\vartheta(\tau,(\ell, \gz))}
{\eta(\tau)}\biggr)^{f(0,\ell)}e^{2\pi i C\omega}
$$
is a (meromorphic) Jacobi form of weight $k=\frac{f(0,0)}2$ 
and index $C$ for the lattice $L$ with a character 
founded in Corollary 2.3 times $v_\eta^{f(0,0)}$.

Next we consider the second term of $\cB_\varphi$ containing the factors with $m>0$  
$$
\hbox{log\,}\biggl(
\prod_{ (n,\ell,m),\, m>0}
(1-q^nr^\ell s^{m})^{f(nm,\ell)}\biggr)=
-\sum_{(n,\ell,\,m)> 0}
f(nm,\ell)\ \sum_{e\ge 1}\frac 1{e}\,q^{en}r^{e\ell}s^{em}=
$$
$$
-\sum_{(a,\vec b,c)>0}
\ \bigl(\sum_{d|(a,\vec b,c)}
d^{-1}f(\frac{ac}{d^2},\,\frac{\vec b}{d})\bigr)
\,q^ar^{\vec b} s^{c}=
 -\sum_{m\ge 1}
m^{-1}\widetilde{\varphi}\,|\, T_-(m)(Z)
$$
where $T_-(m)$ is  the Hecke operator defined in 
\cite{G1} and \cite[(2.12)]{G2}. $T_-(m)$ is a Hecke operator of type $V_m$
from  \cite{EZ} and it  multiplies the index
of Jacobi modular forms by $m$ (see \cite[Corollary 2.9]{G2}).
Therefore we have the following representation of the Borcherds product  
\begin{equation}\label{eq-BHecke}
\cB_{\varphi}(Z)=\widetilde\psi_{L;C}(Z)\exp{\bigl( -\sum_{m\ge 1}
m^{-1}\widetilde{\varphi}| T_-(m)(Z)\bigr)}
\end{equation}
which is similar to \cite[(2.7)]{GN4} where we considered the case of
Siegel modular forms. 
Therefore $\cB_\varphi$ transforms like a modular form of weight $k=\frac{f(0,0)}2$
with respect to the Jacobi modular group $\Gamma^J(L)<\widetilde\Orth^+(L_2)$. 

We know (see \cite[\S 3]{G2}) that 
$\widetilde\Orth^+(L_2)=\langle \Gamma^J(L), V\rangle$ where 
$V : (\tau,\gz, \omega)\to (\omega,\gz, \tau)$.
We can  analyze  the behavior of  the Borcherds product under the  $V$-action
using Proposition 2.2.
A straightforward calculation shows that
$$
\frac{ \cB_\varphi(V\langle Z \rangle)}{\cB_\varphi(Z)}=
\frac{q^Cs^A}{q^As^C}
\prod_{n>0,\,m>0,\,\ell\in L^\vee}
\biggl(\frac{1-s^{-n}r^\ell q^m}{1-q^{-n}r^\ell s^m}\biggr)^{f(-nm,\ell)}.
$$
We note that $f(-nm,\ell)=f(-nm,-\ell)$. Therefore the last product is equal to 
$$
=\frac{q^{C+\sum nf(-nm,\ell)}s^A}{q^As^{C+\sum nf(-nm,\ell)}}
\prod_{n>0,\,m>0,\,\ell\in L^\vee}
\biggl(\frac{s^nr^{-\ell}- q^m}{q^n-r^\ell s^m}\biggr)^{f(-nm,\ell)}=
(-1)^{\sum f(-nm, 0)}
$$
according to the formulae for $C$ from Proposition  \ref{prop-JObst}. 
Therefore we have 
\begin{equation*}\label{eq-Banti}
\cB_\varphi(\tau, \gz,\omega)
=(-1)^{D}\cB_\varphi(\omega, \gz,\tau)\qquad{\rm where }\quad
D=\sum_{n<0} \sigma_0(-n)f(n,0).
\end{equation*}
This proves  that $\cB_\varphi$ transforms like  a modular form 
of weight $\frac{f(0,0)}2$ with respect to $\widetilde\Orth^+(L_2)$.
The character $\chi$ is induced by the $\Gamma^J(L)$-character 
of the Jacobi form $\psi_{L;C}$  and by the last relation  
\begin{equation}\label{eq-char}
\chi|_{\SL_2}=v_\eta^{24A},\ \  
\chi|_{H(L)}([\lambda,\mu;r])=
e^{\pi i C((\lambda,\lambda)+(\mu,\mu)-(\lambda,\mu)+2r)},
\ \ \chi(V)=(-1)^{D}. 
\end{equation}
Now we  calculate the multiplicities of the divisors. 
The Borcherds product $\cB_\varphi$ has an analytic continuation 
to $\cH(L)$ (see \cite[Theorem 5.1 and Theorem 10.1]{B1} 
and \cite[Theorem 13.3]{B2}). The singularities of $\cB_\varphi$
are the solutions  of the equations $(1-q^nr^\ell s^{m}\bigr)^{f(nm,\ell)}=0$
for  $(n, \ell, m)$ with  $2nm-(\ell,\ell)<0$ and $f(nm,\ell)\ne 0$.
The lattice $L$ contains two hyperbolic planes. 
According to the Eichler criterion (see \cite{G2}, \cite{GHS3})
the  $\widetilde\Orth^+(L_2)$-orbit of any  primitive vector $v\in L_2^\vee$
is uniquely determined by 
its norm $(v,v)$ and by its image 
$v\cong \ell \mod L_2$ in the discriminant group $L_2^\vee/L_2$.
Therefore  there exists 
$$
w=(0,n,\ell,1,0)\in \widetilde\Orth^+(L_2)v\quad{\rm such\  that\ }\ 
 (v,v)=2n-(\ell,\ell)<0, \quad  v\cong \ell \mod L_2.
$$
In particular,  $\widetilde\Orth^+(L_2)\cdot \cD_v=
\widetilde\Orth^+(L_2)\langle \{Z\in \cH(L)\,|\,n\tau+(\ell,\gz)+\omega=0\}\rangle.$
The Fourier coefficient $f(n,\ell)$ also  depends only on
the norm  $2n-\ell^2$ and $\ell$ modulo $L$.
Therefore the multiplicity along the divisor $\cD_v$ with 
the vector $v$ as above is equal to 
\begin{equation}\label{eq-mult}
{\rm mult}\  \cD_{v}=
 \sum_{\substack{d>0 \vspace{1pt}\\ 
 (v,v)=2n-(\ell,\ell)\vspace{1pt} \\
 v\equiv \ell \mod L_2}} f(d^2n, d\ell).
\end{equation}
\end{proof}

\begin{corollary}\label{cor-FJexp}
The representation (\ref{eq-BHecke}) gives the first terms
of the Fourier-Jacobi expansion of $\cB_\varphi$
at the one dimensional cusp defined by the lattice $L$
$$
B_\varphi(\tau,\gz,\omega)=
\psi(\tau,\gz)e^{2\pi i C\omega}
\biggl(1-\varphi(\tau,\gz)e^{2\pi i \omega}
+\frac{1}2\bigl(\varphi^2(\tau,\gz)-\varphi(\tau,\gz)|T_{-}(2)\bigr)e^{4\pi i \omega}
$$
\begin{equation*}
-\frac{1}6\bigl(\varphi^3(\tau,\gz)-3\varphi(\tau,\gz)\cdot 
\bigl(\varphi(\tau,\gz)|T_{-}(2)\bigr)
+2\varphi(\tau,\gz)|T_{-}(3)\bigr)e^{6\pi i \omega}
+\dots \biggr).
\end{equation*} 
\end{corollary}
The next corollary is evident but it is very useful if one would like
to prove that a Jacobi (additive) lifting has a Borcherds product.
In this context this property  was often used in \cite{GN4}. 

\begin{corollary}\label{cor-FJcrit}
{\rm \bf The Fourier-Jacobi criterion for automorphic products.}
Let us assume that a modular form 
$$
F(\tau,\gz,\omega)=\varphi_m(\tau,\gz)e^{2\pi i m\omega}+
\varphi_{m+1}(\tau,\gz)e^{2\pi i (m+1)\omega}+\dots
$$
has Borcherds product expansion.  Then 
\begin{equation*}
F=\cB_\varphi\qquad {\rm where}\quad
\varphi(\tau,\gz)=-\frac {\varphi_{m+1}(\tau,\gz)}{\varphi_{m}(\tau,\gz)}.
\end{equation*}
\end{corollary}
\medskip

\noindent
{\bf 3.2. The Fourier--Jacobi expansions of the  Borcherds form $\Phi_{12}$.}
We prove Theorem 1.2 as a corollary of Theorem 2.1. 
Let $N(R)$ be the  Niemeier lattice with the root system $R$.
We put
$$
\varphi_{0,N}(\tau, \mathfrak z)
=\Delta(\tau)^{-1}\vartheta_{N(R)}(\tau, \mathfrak{z})
\in J_{0,1}^{wh}(N(R))
$$
where $\vartheta_{N(R)}$ is the Jacobi theta-series of the even unimodular 
lattice $N(R)$ (see Example 1 of \S 2). We have
$$ 
\varphi_{0,N}(\tau, \mathfrak z)=
\sum_{n\ge -1,\, \ell \in N(R)} f(n,\ell)q^nr^\ell=
q^{-1}+24+
\sum_{v\in R}e^{2\pi i (v, \mathfrak z)}+\dots.
$$
The hyperbolic norm of the index of any non-zero Fourier coefficient
of the Jacobi theta-series is equal to zero. Therefore, 
if $f(n,\ell)\ne 0$, then $2n-(\ell, \ell)\ge -2$. Moreover
if $2n-(\ell, \ell)=-2$, then $f(n,\ell)=1$ because 
$f(n,\ell)$ depends only on the norm of the index.
We note that all $-2$-vectors in the even unimodular lattice $2U\oplus N(R)(-1)$ build 
only one orbit with respect to the orthogonal group. 
According to Theorem 3.1 the Borcherds products $\cB_{\varphi_{0,N}}$  vanishes with 
order one along all rational quadratic divisors $\cD_v$ where 
$v\in 2U\oplus N(R)(-1)$ and $(v,v)=-2$.
Therefore $\cB_{\varphi_{0,N}}$ is equal, up to a constant, to $\Phi_{12}$
according to the Koecher principle.
We can use 
Corollary 3.2 in order to calculate the first two terms in the Fourier--Jacobi
expansion.
If $R=\emptyset$, then $N(\emptyset)=\Lambda_{24}$ and 
\begin{equation}\label{BLe} 
\cB_{\varphi_{0,\Lambda}}
=\Delta(\tau)-\vartheta_{\Lambda_{24}}(\tau, \mathfrak{z})e^{2\pi i \omega}+
\frac{1}2\bigl(
\vartheta_{\Lambda_{24}}\varphi_{0,\Lambda}
-\Delta(\varphi_{0,\Lambda}|T_{-}(2))\bigr)e^{4\pi i \omega} + \dots
\end{equation}
The last function is equal to $\Phi_{12}(Z)$ because the two modular forms 
have the same 
value at the one dimensional Leech cusp.
If $R$ is not empty then for $(\tau, \gz, \omega)\in \cH(N(R))$
\begin{equation}\label{BNR}
\cB_{\varphi_{0,N}}(\tau, \gz, \omega)=
\end{equation}
$$
 \Delta(\tau)\prod_{v\in R_{+}}
\frac{\vartheta(\tau, (v,\mathfrak{z}))}{\eta(\tau)}\,e^{2\pi i\, h(R)\,\omega}-
\vartheta_N(\tau, \mathfrak{z})\prod_{v\in R_+}
\frac{\vartheta(\tau, (v,\mathfrak{z}))}{\eta(\tau)}\,e^{2\pi i\, (h(R)+1)\,\omega}+
\dots
$$
where the product is taken over all positive roots $v$ in the finite root system $R$.
We note that rank $R=24$ and the first term is 
the  Weyl--Kac denominator function of the affine Lie algebra 
$\hat{\mathfrak g}(R)$. We note that the Fourier coefficient of the Weyl vector
$(A, \vec B, C)$ of the Borcherds product is equal to $1$.
Therefore $\cB_{\varphi_{0,N(R)}}=\pm \Phi_{12}$. The sign depends 
on the choice of the positive roots in the finite root system $R$.
Theorem 1.2 is proved.
\medskip

\noindent
{\bf 3.3. The Jacobi criterion for Borcherds products.} 
In this section we give an illustration of the Jacobi criterion formulated 
in Corollary 3.3. For this end we find the Borcherds product of
the reflective modular form $\Phi_2$  which is the 
``roof'' of the   {\it $4A_1$-tower of reflective modular forms} 
for the root lattices $A_1<2A_1<3A_1<4A_1$ (see \cite{G5}).
The last member of this tower for $A_1$ 
is the classical Igusa modular form  $\Delta_5$.
The details of the Jacobi construction of  $\Phi_2$
were given in \cite[Example 2.4]{CG2}.
The ``direct'' product  of the four theta-series is  
a Jacobi form of index $\frac{1}2$ for the lattice $L=4A_1$
with a character of order $2$
$$
\vartheta_{4A_1}(\tau, \mathfrak{z})
=\vartheta(\tau,z_1)\cdot ... \cdot\vartheta(\tau,z_4)
\in J_{2,\frac{1}2}(4A_1,\, v_\eta^{12}\times v_H(4A_1)).
$$
According to  \cite[Theorem 2.2]{CG2} we have
$$
\Phi_2(Z)=\Lift(\vartheta_{4A_1}(\tau, \mathfrak{z}))
\in M_2(\Orth^+(2U\oplus 4A_1(-1)),\chi_2)
$$
where $\chi_2$ is a character of order $2$ of the full orthogonal
group. 
This fundamental modular  form of singular weight has the following 
Fourier expansion
$$
\Phi_2(Z)=
\sum_{\substack{m\equiv 1 \bmod 2\\ m\geqslant 1}}
m^{-1}({\widetilde\vartheta_{4A_1}}\vert_2T_-^{(2)}(m))(Z)=
$$
$$
\sum_{\substack{ \ell=(l_1,\dots, l_4)\in 4A_1^\vee}}
\ \sum_{\substack{
 n,\,m\in \ZZ_{>0}\\
 \vspace{0.5\jot} n\equiv m\equiv 1\,{\rm mod\,}\ZZ\\
\vspace{0.5\jot}  2nm-(\ell,\ell)=0}}
\sigma_1((n,\ell,m))
\biggl(\frac{-4}{2l_1}\biggr)\dots \biggl(\frac{-4}{2l_4}\biggr)
\,e^{\pi i (n\tau+ (\ell,\mathfrak{z})+m\omega)}
$$
where  
$\sigma_1((n,\ell,m))$
is the sum of the positive divisors of the  greatest common divisor of the vector
$(n,\ell,m)\in (U\oplus A_1(-4))^\vee$. 
According \cite[Proposition 2.1]{CG2} and \cite[(1.27)]{GN4}
$$
3^{-1}\vartheta_{4A_1}\vert_2T_{-}(3)(\tau, \gz)=
3\vartheta_{4A_1}(3\tau,3\gz)
+\frac{1}3
\sum_{b=0}^{2}\vartheta_{4A_1}(\frac{\tau+2b}3,\gz).
$$
The quotient
$$
\phi_{0,\,4A_1}=-\frac{3^{-1}\vartheta_{4A_1}\vert_2\,T_{-}(3)}{\vartheta_{4A_1}}
\in J_{0,1}^w(4A_1)
$$
is a weak Jacobi form of weight $0$ and index $1$ for the lattice $4A_1$.
Its   Fourier coefficient  $f(n,\ell)$ 
depends only on $2n-\ell^2$ and $\ell \bmod 4A_1$.
Applying (\ref{vvmf}) to a week Jacobi form of index one for $4A_1$
we see that
if $f(n,\ell)\ne 0$, then  $2n-\ell^2\ge -2$.
Therefore,  in order to find all Fourier coefficients with negative 
hyperbolic norm of their indices we have to calculate only 
the $q^0$-part of the Fourier expansion of   $\phi_{0,\,4A_1}$.
The first term 
$\frac{\vartheta_{4A_1}(3\tau,3\gz)}{\vartheta_{4A_1}(\tau,\gz)}$
of the quotient 
contains coefficients  $q^n$ with positive $n$. The second term 
is equal, up to the sign, to  
$$
\frac{3^{-1}\sum_{b\bmod 3}\biggl(q^{\frac{1}2}\Prod_{n\ge 1}(1-q^n)^4
\Prod_{i=1}^{4}(1-q^nr_i)(1-q^nr_i^{-1})\biggr)\,\vert
\begin{pmatrix} 1&2b\\0&3
\end{pmatrix}}
{q^{\frac{1}2}\Prod_{n\ge 1}(1-q^n)^4
\Prod_{i=1}^{4}(1-q^nr_i)(1-q^nr_i^{-1})}
$$
where  $r_i=e^{2\pi i z_i}$ and 
the matrix denotes the action $\tau\to \frac{\tau+2b}{3}$.
The straightforward  calculation shows that
$$
\phi_{0, 4A_1}(\tau,\gz)=
r_1+r_2+r_3+r_4+4+r_1^{-1}+r_2^{-1}+r_3^{-1}+r_4^{-1}+q(\dots).
$$
This part of the  Fourier expansion contains all orbits of the 
Fourier coefficients with negative norm of its indices. 
Therefore the Borcherds product
$$
\cB_{\phi_{0, 4A_1}}\in M_2(\widetilde{\Orth}^+(2U\oplus 4A_1(-1)),\chi_2)
$$
vanishes with order $1$ along the rational quadratic divisors 
$\widetilde{\Orth}^+(2U\oplus 4A_1(-1))$-equivalent to $z_i=0$ ($1\le i\le 4$)
which are the walls of the reflections with respect to the $(-2)$-roots 
of $4A_1(-1)$.
The Jacobi lifting $\Phi_2$ vanishes along these divisors.
Therefore $\Phi_2/\cB_{\phi_{0, 4A_1}}$ is holomorphic.
Analyzing the first Fourier-Jacobi coefficients we get 
\begin{equation}\label{Phi2}
\Phi_2=\Lift(\vartheta_{4A_1})=\cB_{\phi_{0, 4A_1}}
\end{equation}
due to the Koecher principle.
$\Phi_2$ is a reflective modular form with the simplest possible divisor
in the sense of \cite{G5}. Taking its quasi pullbacks on the lattices
$A_1<2A_1<3A_1<4A_1$ we get three other reflective modular forms with respect  
to  ${\Orth}^+(2U\oplus mA_1(-1))$ for $m=3$, $2$ and $1$.
The last one (for $m=1$) is the Igusa modular form $\Delta_5$, i.e.
$$
\Phi_2(\tau,z_1,z_2,z_3,z_4,\omega)\,|^{(quasi\ pullback)}_{z_2=z_3=z_4=0}=
\Delta_5(\tau,z_1,\omega)\in S_5(\Sp_2(\ZZ),\chi_2).
$$ 
See \cite{G5} for more details where we analyzed also the towers
of reflective modular forms for  $D_1<\dots <D_8$ and $A_2<2A_2<3A_2$.
\medskip

\noindent
{\bf Acknowledgements:}
This work was supported by the grant ANR-09-BLAN-0104-01.
The author is grateful to the Max-Planck-Institut
f\"ur  Mathematik in Bonn for support and for providing excellent 
research atmosphere.

\bibliographystyle{alpha}

\begin{thebibliography}{GHS2}

\bibitem[BB]{BB} W.L.~Baily Jr., A.~Borel, {\it Compactification of
    arithmetic quotients of bounded symmetric domains.} Ann. of Math.
 {\bf 84} (1966), 442--528.


\bibitem[B1]{B1} R.E. Borcherds, {\it Automorphic forms on $O_{s+2,2}(R)$
and infinite products.}
Invent. Math. {\bf 120} (1995), 161--213.

\bibitem[B2]{B2} R.E.~Borcherds,
{\it Automorphic forms with singularities on Grassmannians.}
Invent. Math. {\bf 132} (1998), 491--562.


\bibitem[CG1]{CG1} F.~Cl\'ery,   V.~Gritsenko,
{\it Siegel modular forms with the simplest divisor.}
Proc. London Math. Soc. {\bf 102} (2011),  1024--1052.

\bibitem[CG2]{CG2} F.~Cl\'ery,   V.~Gritsenko
{\it Modular forms of orthogonal type and Jacobi theta-series.}
{\tt arXiv:1106.4733}, 28 pp.

\bibitem[CS]{CS} J.H.~Conway, N.J.A.~Sloane, {\it Sphere packings,
  lattices and groups.} Grundlehren der mathematischen
  Wissenschaften {\bf 290}. Springer-Verlag, New York, 1988. 





\bibitem[EZ]{EZ} M.~Eichler, D.~Zagier, {\it The theory of Jacobi
forms.} Progress in Mathematics {\bf 55}. Birkh\"auser, Boston,
Mass., 1985.

\bibitem[FF]{FF}
A.J.~Feingold, I.B.~Frenkel,
{\it A hyperbolic Kac--Moody algebra and the theory of
Siegel modular forms of genus $2$}.
Math. Ann.
{\bf 263} (1983), 87--144.

\bibitem[GG]{GG} B.~Grandp\^ierre, V.~Gritsenko,
{\it The baby functions of the Borcherds form} $\Phi_{12}$. 
In preparation.

\bibitem[G1]{G1} V.~Gritsenko,
{\it Jacobi functions of n-variables.}
Zap. Nauk. Sem. LOMI
{\bf 168} (1988), 32--45;
English transl. in J. Soviet Math\. {\bf 53}
(1991), 243--252.

\bibitem[G2]{G2} V.~Gritsenko, {\it Modular forms and moduli spaces of
abelian and $\Kthree$ surfaces.} Algebra i Analiz {\bf 6} (1994),
65--102; English translation in St. Petersburg Math. J. {\bf 6}
(1995), 1179--1208.

\bibitem[G3]{G3} V.~Gritsenko,
{\it Elliptic genus of Calabi--Yau manifolds and Jacobi and
Siegel modular forms.}
St. Petersburg Math. J. {\bf 11} (1999), 100--125.

\bibitem[G4]{G4} V.~Gritsenko, 
{\it Complex vector bundles and Jacobi forms.}
Proc. of RIMS Symposium
{\it ``Automorphic forms"}, {\bf 1103},
Kokyuroku, Kyoto,  1999, pp. 71--86; {\tt arXiv:9906191}.


\bibitem[G5]{G5} V.~Gritsenko, \emph{Reflective modular forms in
  algebraic geometry.}  {\tt arXiv:1005.3753}, 28 pp.
  
\bibitem[GH]{GH} V.~Gritsenko, K.~Hulek, 
\emph{Uniruledness of orthogonal modular varieties}. {\tt arXiv:1202.3361},
14 pp.

\bibitem[GHS1]{GHS1} V.~Gritsenko, K.~Hulek, G.K.~Sankaran, {\it The
  Kodaira dimension of the moduli of K3 surfaces}. Invent. Math.
  {\bf 169} (2007), 519--567.

\bibitem[GHS2]{GHS2} V.~Gritsenko, K.~Hulek, G.K.~Sankaran,
  \emph{Moduli spaces of irreducible symplectic manifolds.}
   Compos. Math. {\bf 146} (2010), 404--434.
   
\bibitem[GHS3]{GHS3} V.~Gritsenko, K.~Hulek, G.K.~Sankaran,
\emph{Abelianisation of orthogonal groups and the fundamental group of modular
varieties.} J. of Algebra  {\bf 322}  (2009), 463--478.
   
\bibitem[GN1]{GN1} V.~Gritsenko, V.~Nikulin,
\emph{Siegel automorphic form correction of some Lorentzi\-an
Kac--Moody Lie algebras.}
Amer. J. Math. {\bf 119} (1997), 181--224.

\bibitem[GN2]{GN2} V.~Gritsenko, V.~Nikulin,
{\it The Igusa modular forms and ``the simplest''
Lorentzian Kac--Moody algebras.}
Matem. Sbornik {\bf 187} (1996), 1601--1643.

\bibitem[GN3]{GN3} V.~Gritsenko, V.~Nikulin,
{\it Automorphic forms and Lorentzian Kac-Moody algebras.  I}.
International J. Math.
{\bf 9} (1998), 153--200.

\bibitem[GN4]{GN4} V.~Gritsenko, V.~Nikulin, {\it Automorphic forms and Lorentzian
Kac-Moody algebras.  II.} International J. Math. {\bf 9}  (1998),
201--275.

\bibitem[GN5]{GN5} V.~Gritsenko, V.~Nikulin, 
{\it On classification of Lorentzian Kac--Moody
algebras.}
Russian Math. Survey  {\bf 57} (2002), 921--981.

\bibitem[Ka]{Ka}
V.~Kac,
{\it Infinite dimensional Lie algebras.} Cambridge Univ. Press, 1990.

\bibitem[KP]{KP} V.~Kac, D.H.~Peterson,
{\it Infinite-dimensional Lie algebras, theta functions and modular forms.} Adv. in Math. {\bf 53} (1984), 125--264.


\end{thebibliography}

\bigskip
\noindent
University Lille 1\\
Laboratoire Paul Painlev\'e\\
F-59655 Villeneuve d'Ascq, Cedex\\
France\\
{\tt Valery.Gritsenko@math.univ-lille1.fr}
\end{document}